\documentclass[11pt,a4paper]{amsart}

\usepackage{latexsym}
\usepackage{amsmath}
\usepackage{amsthm}
\usepackage{latexsym}
\usepackage{amssymb}
\usepackage[a4paper , lmargin = {3cm} , rmargin = {3cm} , tmargin = {2.5cm} , bmargin = {2.5cm} ]{geometry}

\usepackage{hyperref}
\usepackage{tikz}

  	\newcommand{\Z}{\ensuremath{\mathbb{Z}}}

\theoremstyle{plain}

\newtheorem*{NewTheoremA}{Theorem A}
\newtheorem*{NewPropositionB}{Proposition B}

\newtheorem{theorem}{Theorem}[section]
\newtheorem{lemma}[theorem]{Lemma}

\title{Planarity of Cayley graphs of graph products of groups}
\author{Olga Varghese}
\thanks{{Funded by the Deutsche 
Forschungsgemeinschaft (DFG, German Research Foundation) under Germany's 
Excellence Strategy -EXC 2044-, Mathematics M\"unster: Dynamics-Geometry-Structure}}
\date{\today}
\address{Olga Varghese\\
Department of Mathematics\\
M\"unster University\\ 
Einsteinstra\ss e 62\\
48149 M\"unster (Germany)}
\email{olga.varghese@uni-muenster.de}

\begin{document}

\pagenumbering{arabic}
\begin{abstract}
We obtain a complete classification of graph products of finite abelian groups whose Cayley graphs with respect to the standard presentations are planar.  
\end{abstract}
\maketitle

\section{Introduction}

This article is located in the intersection of graph theory and group theory. One is interested in understanding which Cayley graphs  can be drawn in the Euclidean plane in such a way that pairs of edges intersect only at vertices. These graphs are called planar.  

All finite groups with planar Cayley graphs were classified by Maschke in \cite{Maschke}. We are interested in understanding which Cayley graphs of graph products of finite abelian groups with the standard generating set are planar. More precisely, let $\Gamma=(V,E)$ be a finite simplicial graph. A vertex labeling on $\Gamma$  is a map $\varphi:V\rightarrow \left\{\text{non-trivial finite abelian groups}\right\}$. A graph $\Gamma$ with a vertex labeling is called a graph product graph. The graph product of groups  $G(\Gamma)$ is the group obtained from the free product of the $\varphi(v)$, by adding the commutator relations $[g,h] = 1$ for all $g\in\varphi(v), h\in\varphi(w))$ such that $\left\{v,w\right\}\in E$. These groups were introduced by Baudisch in \cite{Baudisch} for $\varphi(v)\cong\mathbb{Z}$ for all $v\in V$ and later by Green for arbitrary vertex groups \cite{Green}. For example, free and direct products of finite abelian groups are graph products of groups where the graph product graph is discrete resp. complete. If all $\varphi(v)$ are of order two, then the graph product of groups $G(\Gamma)$ is a right angled Coxeter group.
Let $G(\Gamma)$ be a graph product of finite cyclic groups, i. e. for each vertex $v\in V$ the group $\varphi(v)$ is a finite cyclic group. For each vertex $v\in V$ we choose $a_v\in\varphi(v)$ such that $\langle a_v\rangle=\varphi(v)$ and we denote the Cayley graph of $G(\Gamma)$ associated to the generating set $\left\{a_v\mid v\in V\right\}$ by ${\rm Cay}(G(\Gamma))$. 

It is natural to ask how the shape of $\Gamma$ affects the planarity of ${\rm Cay}(G(\Gamma))$.  For right angled Coxeter groups this has been answered by Droms in \cite[Theorem 1]{Droms}: Let $G(\Gamma)$ be a right angled Coxeter group. The Cayley graph ${\rm Cay}(G(\Gamma))$ is planar iff $\Gamma$ is outerplanar, i.e. $\Gamma$ has no subdivision of $K_4$ or $K_{2,3}$ as a subgraph.

\begin{figure}[h]
\begin{center}
\begin{tikzpicture}
\draw[fill=black]  (0,0) circle (2pt);
\draw[fill=black]  (2,0) circle (2pt);
\draw[fill=black]  (0,2) circle (2pt);
\draw[fill=black]  (2,2) circle (2pt);
\draw (0,0)--(2,0);
\draw (0,0)--(0,2);
\draw (2,0)--(2,2);
\draw (0,2)--(2,2);
\draw (0,0)--(2,2);
\draw (0,2)--(2,0);
\node at (1,-0.5) {$K_4$};

\draw[fill=black]  (4,0.5) circle (2pt);
\draw[fill=black]  (4,1.5) circle (2pt);
\draw[fill=black]  (6,0) circle (2pt);
\draw[fill=black]  (6,1) circle (2pt);
\draw[fill=black]  (6,2) circle (2pt);

\draw (4,0.5)--(6,0);
\draw (4,0.5)--(6,1);
\draw (4,0.5)--(6,2);
\draw (4,1.5)--(6,0);
\draw (4,1.5)--(6,1);
\draw (4,1.5)--(6,2);
\node at (5,-0.5) {$K_{2,3}$};
\end{tikzpicture}
\end{center}
\end{figure}

\newpage
We obtain the following characterization for graph products of finite cyclic groups. 
\begin{NewTheoremA}
Let $\Gamma$ be a graph product graph of finite cyclic groups. Let $\Gamma_{=2}$ be the subgraph generated by its vertices labelled with groups of order $2$ and $\Gamma_{>2}$ the subgraph generated by its vertices labelled with groups of order $>2$. The Cayley graph ${\rm Cay}(G(\Gamma))$ is planar if and only if:
\begin{enumerate}
\item[(i)] $\Gamma_{=2}$ is outerplanar
\item[(ii)] $\Gamma_{>2}$ is a discrete graph
\item[(iii)] If $v\in\Gamma_{>2}$ is a vertex, then the subgraph of $\Gamma$ generated by its vertices adjacent to $v$  is empty or consists of one vertex or of two disjoint vertices.
\item[(iv)] If $\Gamma'\subseteq \Gamma$ is an induced cycle, then $\Gamma'\subseteq\Gamma_{=2}$.
\end{enumerate} 
\end{NewTheoremA}
For example, the Cayley graph of the graph product of groups associated to the following graph product graph is planar.

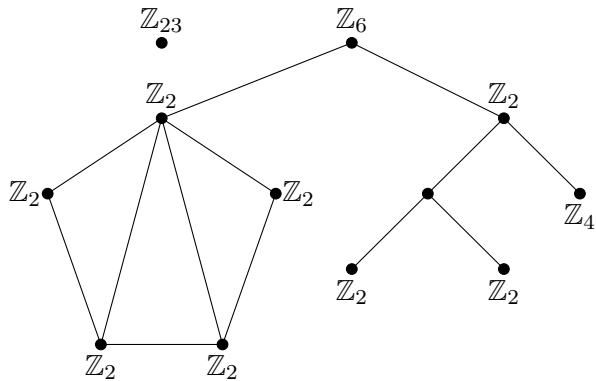
\begin{figure}[h]
\begin{center}
\begin{tikzpicture}
\draw[fill=black]  (0.7,0) circle (2pt);
\draw[fill=black]  (2.3,0) circle (2pt);
\draw[fill=black]  (0,2) circle (2pt);
\draw[fill=black]  (3,2) circle (2pt);
\draw[fill=black]  (1.5,3) circle (2pt);
\draw (0.7,0)--(2.3,0);
\draw (0,2)--(0.7,0);
\draw (0,2)--(1.5,3);
\draw (1.5,3)--(3,2);
\draw (2.3,0)--(3,2);
\draw (0.7,0)--(1.5,3);
\draw (2.3,0)--(1.5,3);

\draw[fill=black]  (1.5,4) circle (2pt);

\draw[fill=black]  (4,4) circle (2pt);
\draw[fill=black]  (6,3) circle (2pt);
\draw (1.5,3)--(4,4);
\draw (6,3)--(4,4);

\draw[fill=black]  (5,2) circle (2pt);
\draw[fill=black]  (7,2) circle (2pt);
\draw[fill=black]  (4,1) circle (2pt);
\draw[fill=black]  (6,1) circle (2pt);
\draw (5,2)--(6,3);
\draw (7,2)--(6,3);
\draw (4,1)--(5,2);
\draw (6,1)--(5,2);

\node at (0.7,-0.3) {$\mathbb{Z}_2$};
\node at (2.3,-0.3) {$\mathbb{Z}_2$};
\node at (-0.3,2) {$\mathbb{Z}_2$};
\node at (3.3,2) {$\mathbb{Z}_2$};

\node at (1.5,3.3) {$\mathbb{Z}_2$};
\node at (1.5,4.3) {$\mathbb{Z}_{23}$};

\node at (4,0.7) {$\mathbb{Z}_2$};
\node at (6,0.7) {$\mathbb{Z}_2$};
\node at (7,1.7) {$\mathbb{Z}_4$};
\node at (6,3.3) {$\mathbb{Z}_2$};
\node at (4,4.3) {$\mathbb{Z}_6$};
\end{tikzpicture}
\end{center}
\end{figure}

We want to remark that the graph product  of finite abelian groups $G(\Gamma)$ is isomorphic to the graph product of groups $G(\Gamma')$ obtained by replacing each vertex $v\in\Gamma$  by a complete graph with vertices labelled by the cyclic summands of $\varphi(v)$. Thus, Theorem A gives a complete characterization of graph products of finite abelian groups whose Cayley graphs with the standard generating sets are planar.

Our proof of Theorem A involves three ingredients. First we prove that some  Cayley graphs of 'small' graph products of groups are planar and some not. Next, we need some group constructions preserving planarity which were proven in \cite[Theorem 3]{Arzhantseva}. In the third step, we build the whole graph product of groups $G(\Gamma)$ out of subgroups $G(\Gamma')$, whose Cayley graphs are planar, using group constructions preserving planarity.

Planarity of a Cayley graph has algebraic consequences for the group. In particular, we deduce from Droms, Servatius and Babai results in \cite{Babai}, \cite{Droms2}, \cite{DromsServatius} 
the following consequence.

\begin{NewPropositionB}
Let $G$ be a finitely generated group. If $G$ has a planar Cayley graph, then $G$ is coherent (i.e. every finitely generated subgroup of $G$ is finitely presented).
\end{NewPropositionB} 
Let us mention that the converse of the result in Proposition B is not true, as we can consider the group $\mathbb{Z}_3\times\mathbb{Z}_3$ which is coherent but does not have planar Cayley graph.

\section{Graphs}
In this section we briefly present the main definitions and properties concerning simplicial and directed graphs, in particular Cayley graphs. A detailed description of these graphs and their properties can be found in 
 \cite{Diestel}, \cite{Harris} or \cite{Serre}.
 
\subsection{Simplicial graphs}
A {\it simplicial graph} $\Gamma=(V,E)$ consists of a set $V\neq\emptyset$ and a set $E$ of $2$-element subsets of $V$. The elements of $V$ are called {\it vertices} and the elements of $E$ are its {\it edges}.  If $V'\subseteq V$ and $E'\subseteq E$, then $\Gamma'=(V', E')$ is called a {\it subgraph} of $\Gamma$. If $\Gamma'$ is a subgraph of $\Gamma$ and $E'$ contains all the edges $\left\{v, w\right\}\in E$ with $v, w\in V'$, then $\Gamma'$ is called an {\it induced subgraph} of $\Gamma$. For $V'\subseteq V$ we denote by $\langle V'\rangle$ the smallest induced subgraph of $\Gamma$ with $(V',\emptyset)\subseteq \langle V'\rangle$. This subgraph is called a graph {\it generated} by $V'$. A {\it path} of length $n$ is a graph $P_n=(V, E)$ of the form  $V=\left\{v_0, \ldots, v_n\right\}$ and $E=\left\{\left\{v_0, v_1\right\}, \left\{v_1, v_2\right\}, \ldots, \left\{v_{n-1}, v_n\right\}\right\}$ where the $v_i$, $0\leq i\leq n$, are pairwise distinct. If $P_n=(V,E)$ is a path of length $n\geq 3$, then the graph $C_{n+1}:=(V, E\cup\left\{\left\{v_n, v_0\right\}\right\})$ is called a {\it cycle} of length $n+1$.
The {\it complete graph} $K_n$ is the graph with $n$ vertices and an edge for every pair of vertices. The {\it complete bipartite graph} $K_{n,m}$ is the graph such that the vertex set $V$ of $K_{n,m}$ is a disjoint union of $V_1$ and $V_2$ of cardinality $n$ resp. $m$ such that for every $v\in V_1$ and $w\in V_2$ $\left\{v, w\right\}$ is an edge of $K_{n,m}$ and the subgraph generated by the vertex set $V_1$ resp. $V_2$ is discrete. A graph $\Gamma=(V,E)$ is called {\it connected} if any two vertices $v, w\in V$ are contained in a subgraph $\Gamma'$ of $\Gamma$ such that $\Gamma'$ is a path. A maximal connected subgraph of $\Gamma$ is called a {\it connected component} of $\Gamma$. A {\it subdivision of a graph} $\Gamma$ is a graph $\Gamma'$ obtained from $\Gamma$ by replacing edges through paths of finite lenght.

\subsection{Directed graphs}
A {\it directed graph} consists of two sets, the set of vertices $V\neq\emptyset$ and the set of edges $E$ with two functions: $\iota:E\rightarrow V$ and $\tau:E\rightarrow V$. The vertex $\iota(e)$ is called the {\it initial vertex} of $e\in E$ and $\tau(e)$ is called its {\it terminal vertex}.

Let $G$ be a group with finite generating set $S$. We call a directed graph ${\rm Cay}(G,S)$ with vertex set $G$ and  edge set $G\times S$, where the edge $(g, s)$ has initial vertex $g$ and terminal vertex $gs$ the {\it Cayley graph} of $G$ associated to the generating set $S$. 
The Cayley graph depends on the choice of the generating set of the group. 
For example ${\rm Cay}(\mathbb{Z}_5,\left\{ 1\right\})$  and ${\rm Cay}(\mathbb{Z}_5,\left\{1, 2, 3, 4\right\})$ are very different from the graph theoretical point of view.

\subsection{Planarity of graphs}
In practice, a simplicial or directed graph is often represented by a diagram. We draw a point for each vertex $v$ of the graph, and a line resp. directed line joining two points $v$ and $w$ if $\left\{v, w\right\}$ is an edge resp. if there exists an edge $e$ with $\iota(e)=v$ and $\tau(e)=w$.

A graph is called {\it planar} if it can be drawn in the Euclidean plane in such a way that pairs of edges intersect only at vertices. For example, the graph $K_4$ is planar and the graphs $K_5$ and $K_{3,3}$  are non-planar, see \cite[Theorems 1.13, 1.15]{Harris}.

                                                                                                    By a fundamental result of Kuratowski \cite{Kuratowski} every finite non-planar graph contains a subdivision of $K_5$ or $K_{3,3}$ as a subgraph. An analogous result for infinite non-planar graphs is proven by Dirac and Schuster in \cite{Dirac}.

We are interested in planarity of Cayley graphs ${\rm Cay}(G, S)$. These graphs are directed and possibly have two edges between two vertices. More precisely, if $s\in S$ has order $2$ or if $s$ and $s^{-1}$ are both contained in $S$, then there exist two edges between $g$ and $gs$. We can glue these two edges together and this process does not change the planarity. The direction of the edges is also not important for planarity. Thus, we forget the direction of the edges and we draw all Cayley graphs  in this paper with undirected edges and without multiple edges. 

\section{Graph products of groups}
In this section we collect some useful properties concerning graph products of groups. 
\begin{lemma}
\label{subgroup}
Let $\Gamma$ be a graph product graph. If $\Gamma'$ is an induced subgraph of $\Gamma$, then ${\rm Cay}(G(\Gamma'))$ is a subgraph of ${\rm Cay}(G(\Gamma))$.
\end{lemma}
\begin{proof}
It was proven by Green in \cite[3.20]{Green} that the group $G(\Gamma')$ is a subgroup of $G(\Gamma)$ with the canonical generating set. Thus, ${\rm Cay}(G(\Gamma'))$ is a subgraph of ${\rm Cay}(G(\Gamma))$.
\end{proof}

We will need the following group constructions preserving planarity for the proof of Theorem A.
\begin{lemma}
\label{amalgam}
\begin{enumerate}
\item[(i)] Let $\Gamma$ be a graph product graph and let $\Gamma_1, \Gamma_2$ induced subgraphs such that $\Gamma=\Gamma_1\cup\Gamma_2$ and $\Gamma_1\cap\Gamma_2=\emptyset$. If ${\rm Cay}(G(\Gamma_1))$ and ${\rm Cay}(G(\Gamma_2)$ are planar, then ${\rm Cay}(G(\Gamma))$ is planar.
\item[(ii)] Let $\Gamma$ be a graph product graph and let $\Gamma_1, \Gamma_2$ be induced subgraphs such that $\Gamma=\Gamma_1\cup\Gamma_2$ and $\Gamma_1\cap\Gamma_2=(\left\{v\right\}, \emptyset)$ with $\varphi(v)\cong\mathbb{Z}_2$. If ${\rm Cay}(G(\Gamma_1))$ and ${\rm Cay}(G(\Gamma_2))$ are planar, then ${\rm Cay}(G(\Gamma))$ is planar.
\end{enumerate}
\end{lemma}
\begin{proof}
The first statement of the above lemma follows from the fact that the graph product of groups $G(\Gamma)$ is a free product of $G(\Gamma_1)$ and $G(\Gamma_2)$ and by \cite[Theorem 3]{Arzhantseva} planarity is preserved under taking free products.

By the assumption of the second statement we have the following decomposition $G(\Gamma)=G(\Gamma_1)*_{\varphi(v)} G(\Gamma_2)$. By \cite[Theorem 3]{Arzhantseva} it follows that ${\rm Cay}(G(\Gamma))$ is planar.

\end{proof}

For example, the following figure shows a piece of the Cayley graph of $\Z_4*\Z_2$.

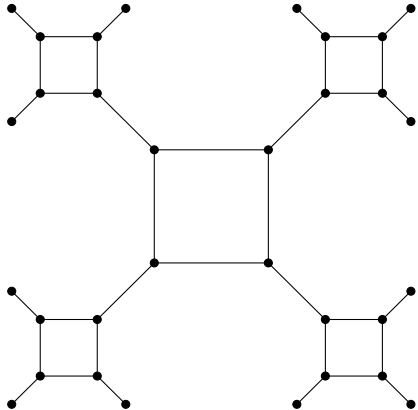
\begin{figure}[h]
\begin{center}
\begin{tikzpicture}[scale=0.75, transform shape]
\draw[fill=black]  (0,0) circle (2pt);
\draw[fill=black]  (2,0) circle (2pt);
\draw[fill=black]  (0,2) circle (2pt);
\draw[fill=black]  (2,2) circle (2pt);
\draw (0,0)--(2,0);
\draw (0,0)--(0,2);
\draw (2,2)--(2,0);
\draw (0,2)--(2,2);

\draw[fill=black]  (3,3) circle (2pt);
\draw[fill=black]  (4,3) circle (2pt);
\draw[fill=black]  (3,4) circle (2pt);
\draw[fill=black]  (4,4) circle (2pt);
\draw (3,3)--(4,3);
\draw (3,3)--(3,4);
\draw (4,3)--(4,4);
\draw (3,4)--(4,4);
\draw (2,2)--(3,3);

\draw[fill=black]  (4.5,2.5) circle (2pt);
\draw[fill=black]  (4.5,4.5) circle (2pt);
\draw[fill=black]  (2.5,4.5) circle (2pt);
\draw (4,3)--(4.5,2.5);
\draw (4,4)--(4.5,4.5);
\draw (3,4)--(2.5,4.5);

\draw[fill=black]  (3,-1) circle (2pt);
\draw[fill=black]  (4,-1) circle (2pt);
\draw[fill=black]  (3,-2) circle (2pt);
\draw[fill=black]  (4,-2) circle (2pt);
\draw (3,-1)--(4,-1);
\draw (4,-1)--(4,-2);
\draw (4,-2)--(3,-2);
\draw (3,-2)--(3,-1);
\draw (2,0)--(3,-1);

\draw[fill=black]  (4.5,-0.5) circle (2pt);
\draw[fill=black]  (4.5,-2.5) circle (2pt);
\draw[fill=black]  (2.5,-2.5) circle (2pt);
\draw (4,-1)--(4.5,-0.5);
\draw (4,-2)--(4.5,-2.5);
\draw (3,-2)--(2.5,-2.5);

\draw[fill=black]  (-1,3) circle (2pt);
\draw[fill=black]  (-2,3) circle (2pt);
\draw[fill=black]  (-1,4) circle (2pt);
\draw[fill=black]  (-2,4) circle (2pt);
\draw (-2,4)--(-1,4);
\draw (-1,4)--(-1,3);
\draw (-1,3)--(-2,3);
\draw (-2,3)--(-2,4);
\draw (0,2)--(-1,3);

\draw[fill=black]  (-0.5,4.5) circle (2pt);
\draw[fill=black]  (-2.5,4.5) circle (2pt);
\draw[fill=black]  (-2.5,2.5) circle (2pt);
\draw (-2,4)--(-2.5,4.5);
\draw (-1,4)--(-0.5,4.5);
\draw (-2,3)--(-2.5,2.5);

\draw[fill=black]  (-2,-1) circle (2pt);
\draw[fill=black]  (-1,-1) circle (2pt);
\draw[fill=black]  (-2,-2) circle (2pt);
\draw[fill=black]  (-1,-2) circle (2pt);
\draw (-2,-1)--(-1,-1);
\draw (-1,-1)--(-1,-2);
\draw (-1,-2)--(-2,-2);
\draw (-2,-2)--(-2,-1);
\draw (-1,-1)--(0,0);

\draw[fill=black]  (-0.5,-2.5) circle (2pt);
\draw[fill=black]  (-2.5,-0.5) circle (2pt);
\draw[fill=black]  (-2.5,-2.5) circle (2pt);
\draw (-2,-1)--(-2.5,-0.5);
\draw (-2,-2)--(-2.5,-2.5);
\draw (-1,-2)--(-0.5,-2.5);
\end{tikzpicture}
\end{center}
\end{figure}

\section{Small graph product graphs}
In this section we discuss some examples of planar and non-planar Cayley graphs of small graph product graphs. 

\begin{lemma}
\label{edge}
Let $\Gamma$ be the following graph product graph:
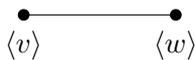
\begin{figure}[h]
\begin{center}
\begin{tikzpicture}
\draw[fill=black]  (0,0) circle (2pt);
\draw[fill=black]  (2,0) circle (2pt);
\draw (0,0)--(2,0);
\node at (0,-0.4) {$\langle v\rangle$}; 
\node at (2,-0.4) {$\langle w\rangle$}; 
 \end{tikzpicture}
\end{center}
\end{figure}

If ${\rm ord}(v)=n\geq 3$ and ${\rm ord}(w)=m\geq3$, then ${\rm Cay}(G(\Gamma))$ is non-planar.
\end{lemma}
\begin{proof}
We show that a subdivision of $K_{3,3}$ is a subgraph of ${\rm Cay}(G(\Gamma))$. Since $K_{3, 3}$ is non-planar and non-planarity is preserved by subdivision of a graph it follows that ${\rm Cay}(G(\Gamma))$ is non-planar.
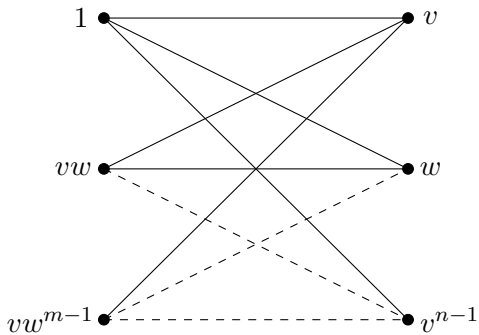
\begin{figure}[h]
\begin{center}
\begin{tikzpicture}
\draw[fill=black]  (0,0) circle (2pt);
\draw[fill=black]  (4,0) circle (2pt);
\draw[fill=black]  (0,2) circle (2pt);
\draw[fill=black]  (4,2) circle (2pt);
\draw[fill=black]  (0,4) circle (2pt);
\draw[fill=black]  (4,4) circle (2pt);

\draw[dashed] (0,0)--(4,0);
\draw[dashed] (0,0)--(4,2);
\draw (0,0)--(4, 4);

\draw[dashed] (0,2)--(4,0);
\draw (0,2)--(4,2);
\draw (0,2)--(4, 4);

\draw (0,4)--(4,0);
\draw (0,4)--(4,2);
\draw (0,4)--(4, 4);

\node at (-0.7,0) {$vw^{m-1}$}; 
\node at (-0.4,2) {$vw$}; 
\node at (-0.3, 4) {$1$}; 

\node at (4.55,0) {$v^{n-1}$}; 
\node at (4.3,2) {$w$}; 
\node at (4.3,4) {$v$};
\end{tikzpicture}
\end{center}
\end{figure}

The dashed paths have the following structure:

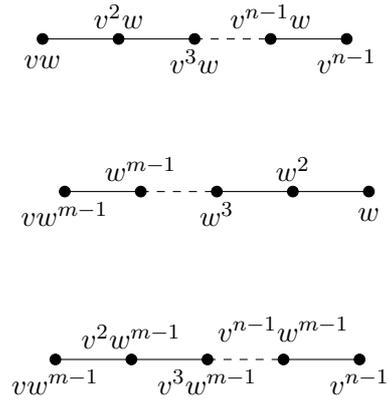
\begin{figure}[h]
\begin{center}
\begin{tikzpicture}
\draw[fill=black]  (0,0) circle (2pt);
\node at (0,-0.3) {$vw$};
\draw[fill=black]  (1,0) circle (2pt);
\node at (1,0.3) {$v^2w$};
\draw[fill=black]  (2,0) circle (2pt);
\node at (2,-0.3) {$v^3w$};
\draw[fill=black]  (3,0) circle (2pt);
\node at (3,0.3) {$v^{n-1}w$};
\draw[fill=black]  (4,0) circle (2pt);
\node at (4,-0.3) {$v^{n-1}$};
\draw (0,0)--(2, 0);
\draw[dashed] (2,0)--(3, 0);
\draw (3,0)--(4, 0);
\end{tikzpicture}
\end{center}
\end{figure}

\begin{figure}[h]
\begin{center}
\begin{tikzpicture}
\draw[fill=black]  (0,0) circle (2pt);
\node at (0,-0.3) {$vw^{m-1}$};
\draw[fill=black]  (1,0) circle (2pt);
\node at (1,0.3) {$w^{m-1}$};
\draw[fill=black]  (2,0) circle (2pt);
\node at (2,-0.3) {$w^3$};
\draw[fill=black]  (3,0) circle (2pt);
\node at (3,0.3) {$w^2$};
\draw[fill=black]  (4,0) circle (2pt);
\node at (4,-0.3) {$w$};
\draw (0,0)--(1, 0);
\draw[dashed] (1,0)--(2, 0);
\draw (2,0)--(4, 0);
\end{tikzpicture}
\end{center}
\end{figure}

\newpage
\begin{figure}[h]
\begin{center}
\begin{tikzpicture}
\draw[fill=black]  (0,0) circle (2pt);
\node at (0,-0.3) {$vw^{m-1}$};
\draw[fill=black]  (1,0) circle (2pt);
\node at (1,0.3) {$v^2w^{m-1}$};
\draw[fill=black]  (2,0) circle (2pt);
\node at (2,-0.3) {$v^3w^{m-1}$};
\draw[fill=black]  (3,0) circle (2pt);
\node at (3,0.4) {$v^{n-1}w^{m-1}$};
\draw[fill=black]  (4,0) circle (2pt);
\node at (4,-0.3) {$v^{n-1}$};
\draw (0,0)--(2, 0);
\draw[dashed] (2,0)--(3, 0);
\draw (3,0)--(4, 0);
\end{tikzpicture}
\end{center}
\end{figure}

Thus, ${\rm Cay}(G(\Gamma))$ is non-planar.
\end{proof}

\begin{lemma}
\label{star}
Let $\Gamma$ be the following graph product graph:  
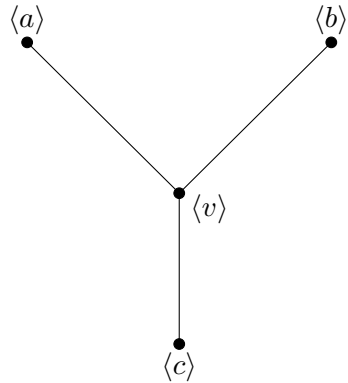
\begin{figure}[h]
\begin{center}
\begin{tikzpicture}
\draw[fill=black]  (0,0) circle (2pt);
\draw[fill=black]  (2,2) circle (2pt);
\draw[fill=black]  (-2,2) circle (2pt);
\draw[fill=black]  (0,-2) circle (2pt);
\draw (0,0)--(2,2);
\draw (0,0)--(-2,2);
\draw (0,0)--(0,-2);
\node at (0,-2.3) {$\langle c\rangle$}; 
\node at (0.4,-0.2) {$\langle v\rangle$}; 
\node at (-2, 2.3) {$\langle a \rangle$}; 
\node at (2,2.3) {$\langle b\rangle$}; 
\end{tikzpicture}
\end{center}
\end{figure}

If ${\rm ord}(a)={\rm ord}(b)={\rm ord}(c)=2$ and ${\rm ord}(v)=n\geq3$, then $Cay(G(\Gamma))$ is non-planar. 
\end{lemma}
\begin{proof}
We show again that a subdivision of $K_{3,3}$ is a subgraph of $Cay(G(\Gamma))$.
\newpage
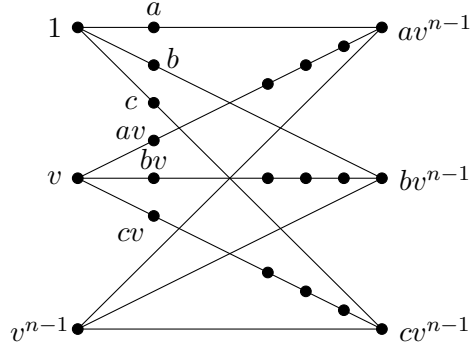
\begin{figure}[h]
\begin{center}
\begin{tikzpicture}
\draw[fill=black]  (0,0) circle (2pt);
\draw[fill=black]  (4,0) circle (2pt);
\draw[fill=black]  (0,2) circle (2pt);
\draw[fill=black]  (4,2) circle (2pt);
\draw[fill=black]  (0,4) circle (2pt);
\draw[fill=black]  (4,4) circle (2pt);

\draw (0,0)--(4,0);
\draw (0,0)--(4,2);
\draw (0,0)--(4, 4);

\draw (0,2)--(4,0);
\draw (0,2)--(4,2);
\draw (0,2)--(4, 4);

\draw (0,4)--(4,0);
\draw (0,4)--(4,2);
\draw (0,4)--(4, 4);

\node at (-0.5,0) {$v^{n-1}$}; 
\node at (-0.3,2) {$v$}; 
\node at (-0.3, 4) {$1$}; 

\node at (4.7,0) {$cv^{n-1}$}; 
\node at (4.7,2) {$bv^{n-1}$}; 
\node at (4.7,4) {$av^{n-1}$};

\draw[fill=black]  (1,4) circle (2pt);
\node at (1,4.25) {$a$};
\draw[fill=black]  (1,3.5) circle (2pt);
\node at (1.25,3.6) {$b$};
\draw[fill=black]  (1,3) circle (2pt);
\node at (0.7,3) {$c$};

\draw[fill=black]  (1,2.5) circle (2pt);
\node at (0.7,2.6) {$av$};
\draw[fill=black]  (1,2) circle (2pt);
\node at (1,2.25) {$bv$};

\draw[fill=black]  (3,2) circle (2pt);

\draw[fill=black]  (2.5,2) circle (2pt);
\draw[fill=black]  (3.5,2) circle (2pt);
\draw[fill=black]  (1,1.5) circle (2pt);
\node at (0.7,1.25) {$cv$};

\draw[fill=black]  (2.5,3.25) circle (2pt);
\draw[fill=black]  (3,3.5) circle (2pt);
\draw[fill=black]  (3.5,3.75) circle (2pt);

\draw[fill=black]  (2.5,0.75) circle (2pt);
\draw[fill=black]  (3,0.5) circle (2pt);
\draw[fill=black]  (3.5,0.25) circle (2pt);

\end{tikzpicture}
\end{center}
\end{figure}

Hence, $Cay(G(\Gamma))$ is non-planar.
\end{proof}

\begin{lemma}
\label{path}
Let $\Gamma$ be the following graph product graph: 

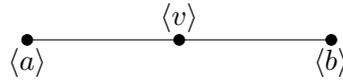
\begin{figure}[h]
\begin{center}
\begin{tikzpicture}
\draw[fill=black]  (0,0) circle (2pt);
\draw[fill=black]  (-2,0) circle (2pt);
\draw[fill=black]  (2,0) circle (2pt);
\draw (-2,0)--(2,0);
\node at (2,-0.3) {$\langle b\rangle$}; 
\node at (0,0.3) {$\langle v\rangle$}; 
\node at (-2, -0.3) {$\langle a \rangle$}; 
\end{tikzpicture}
\end{center}
\end{figure}

If ${\rm ord}(a)={\rm ord}(b)=2$ and ${\rm ord}(v)=n\geq 3$, then ${\rm Cay}(G(\Gamma))$ is planar. 
\end{lemma}
\begin{proof}
A planar drawing of this Cayley graph is given for $n=6$ in \cite[\S 3, Fig. 1]{Mohar}. It is easy to verify how this Cayley graph should be drawn in the plane for an arbitrary $n$.
\end{proof}

\begin{lemma}
\label{cycle}
Let $\Gamma=(V, E)$ be a graph product graph. If $\Gamma$ is a cycle and if there exists $v\in V$ with $\#\varphi(v)=n\geq 3$, then ${\rm Cay}(G(\Gamma))$ is non-planar.
\end{lemma}
\begin{proof}

If $\Gamma$ has the following shape
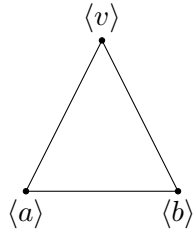
\begin{figure}[h]
\begin{center}
\begin{tikzpicture}
\draw[fill=black]  (0,0) circle (1pt);
\draw[fill=black]  (2,0) circle (1pt);
\draw[fill=black]  (1,2) circle (1pt);

\draw (0,0)--(2,0);
\draw (0,0)--(1,2);
\draw (1,2)--(2,0);

\node at (0,-0.3) {$\langle a\rangle$}; 
\node at (2,-0.3) {$\langle b\rangle$}; 
\node at (1, 2.3) {$\langle v \rangle$}; 
\end{tikzpicture}
\end{center}
\end{figure}

we consider two cases: if ${\rm ord}(a)$ or ${\rm ord}(b)\geq 3$, then it follows from Lemma \ref{subgroup} that ${\rm Cay}(\langle a, v\mid a^m, v^n, av=va\rangle)$ where $m={\rm ord}(a)\geq 3$ is a subgraph of ${\rm Cay}(G(\Gamma))$. By Lemma \ref{edge}  ${\rm Cay}(\langle a, v\mid a^m, v^n, av=va\rangle)$ is non-planar, hence ${\rm Cay}(G(\Gamma))$ is non-planar.

If ${\rm ord}(a)={\rm ord}(b)=2$ we again construct  a subdivision of $K_{3,3}$ as a subgraph of ${\rm Cay}(G(\Gamma))$.

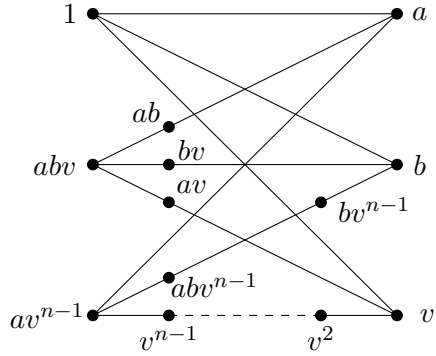
\begin{figure}[h]
\begin{center}
\begin{tikzpicture}
\draw[fill=black]  (0,0) circle (2pt);
\draw[fill=black]  (4,0) circle (2pt);
\draw[fill=black]  (0,2) circle (2pt);
\draw[fill=black]  (4,2) circle (2pt);
\draw[fill=black]  (0,4) circle (2pt);
\draw[fill=black]  (4,4) circle (2pt);

\draw (0,0)--(1,0);
\draw[dashed] (1,0)--(3,0);
\draw (3,0)--(4,0);

\draw (0,0)--(4,2);
\draw (0,0)--(4, 4);

\draw (0,2)--(4,0);
\draw (0,2)--(4,2);
\draw (0,2)--(4, 4);

\draw (0,4)--(4,0);
\draw (0,4)--(4,2);
\draw (0,4)--(4, 4);

\node at (-0.6,0) {$av^{n-1}$}; 
\node at (-0.5,2) {$abv$};

\draw[fill=black]  (1,2) circle (2pt);
\node at (1.3,2.2) {$bv$};

\draw[fill=black]  (1,2.5) circle (2pt);
\node at (0.7,2.7) {$ab$};

\draw[fill=black]  (1,1.5) circle (2pt);
\node at (1.3,1.7) {$av$};

\node at (-0.3, 4) {$1$}; 
\node at (4.4,0) {$v$}; 
\node at (4.3,2) {$b$}; 
\node at (4.3,4) {$a$};

\draw[fill=black]  (3,0) circle (2pt);
\node at (3,-0.3) {$v^{2}$};

\draw[fill=black]  (1,0) circle (2pt);
\node at (1,-0.3) {$v^{n-1}$};

\draw[fill=black]  (1,0.5) circle (2pt);
\node at (1.6,0.4) {$abv^{n-1}$};

\draw[fill=black]  (3,1.5) circle (2pt);
\node at (3.7,1.4) {$bv^{n-1}$};
\end{tikzpicture}
\end{center}
\end{figure}
\newpage

If $\Gamma$ is a cycle of lenght $\geq 4$
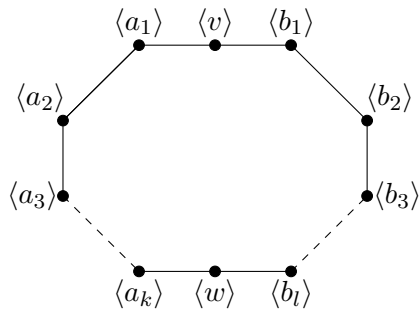
\begin{figure}[h]
\begin{center}
\begin{tikzpicture}

\draw[fill=black]  (1,0) circle (2pt);
\draw[fill=black]  (2,0) circle (2pt);
\draw[fill=black]  (3,0) circle (2pt);

\draw[fill=black]  (0,1) circle (2pt);
\draw[fill=black]  (4,1) circle (2pt);

\draw[fill=black]  (0,2) circle (2pt);
\draw[fill=black]  (4,2) circle (2pt);

\draw[fill=black]  (1,3) circle (2pt);
\draw[fill=black]  (2,3) circle (2pt);
\draw[fill=black]  (3,3) circle (2pt);

\draw (1,3)--(3,3);
\draw (1,3)--(0,2);
\draw (0,2)--(0,1);

\draw[dashed] (0,1)--(1,0);
\draw (1,3)--(0,2);
\draw[dashed] (3,0)--(4,1);
\draw (4,1)--(4,2);
\draw (4,2)--(3,3);

\draw (1,0)--(3,0);

\node at (2, 3.3) {$\langle v \rangle$}; 
\node at (1, 3.3) {$\langle a_1 \rangle$}; 
\node at (3, 3.3) {$\langle b_1 \rangle$};
\node at (-0.3, 2.3) {$\langle a_2 \rangle$}; 
\node at (4.3, 2.3) {$\langle b_2 \rangle$};
\node at (-0.4, 1) {$\langle a_3 \rangle$}; 
\node at (4.4, 1) {$\langle b_3 \rangle$};

\node at (2, -0.3) {$\langle w \rangle$};

\node at (1, -0.3) {$\langle a_k \rangle$}; 
\node at (3, -0.3) {$\langle b_l \rangle$};

\end{tikzpicture}
\end{center}
\end{figure}

then we consider again two cases: if ${\rm ord}(a_1)$ or ${\rm ord}(b_1)\geq 3$, then it follows from Lemma \ref{subgroup} that ${\rm Cay}(\langle a_1, v\mid a_1^m, v^n, a_1v=va_1\rangle)$ where $m={\rm ord}(a_1)\geq 3$ is a subgraph of ${\rm Cay}(G(\Gamma))$. By Lemma \ref{edge}  ${\rm Cay}(\langle a_1, v\mid a_1^m, v^n, a_1v=va_1\rangle)$ is non-planar, hence ${\rm Cay}(G(\Gamma))$ is non-planar.

If ${\rm ord}(a_1)={\rm ord}(b_1)=2$ we construct a subdivision of $K_{3,3}$ as a subgraph of ${\rm Cay}(G(\Gamma))$:

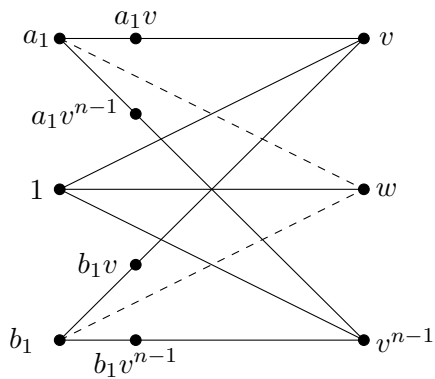
\begin{figure}[h]
\begin{center}
\begin{tikzpicture}
\draw[fill=black]  (0,0) circle (2pt);
\draw[fill=black]  (4,0) circle (2pt);
\draw[fill=black]  (0,2) circle (2pt);
\draw[fill=black]  (4,2) circle (2pt);
\draw[fill=black]  (0,4) circle (2pt);
\draw[fill=black]  (4,4) circle (2pt);

\draw (0,0)--(4,0);
\draw[dashed] (0,0)--(4,2);
\draw (0,0)--(4, 4);

\draw (0,2)--(4,0);
\draw (0,2)--(4,2);
\draw (0,2)--(4, 4);

\draw (0,4)--(4,0);
\draw[dashed] (0,4)--(4,2);
\draw (0,4)--(4, 4);

\node at (-0.5,0) {$b_1$}; 
\node at (-0.3,2) {$1$}; 
\node at (-0.3, 4) {$a_1$}; 

\node at (4.55,0) {$v^{n-1}$}; 
\node at (4.3,2) {$w$}; 
\node at (4.3,4) {$v$};

\draw[fill=black]  (1,4) circle (2pt);
\node at (1,4.25) {$a_1v$};
\draw[fill=black]  (1,3) circle (2pt);
\node at (0.2,3) {$a_1v^{n-1}$};

\draw[fill=black]  (1,1) circle (2pt);
\node at (0.5,1) {$b_1v$};
\draw[fill=black]  (1,0) circle (2pt);
\node at (1,-0.3) {$b_1v^{n-1}$};
\end{tikzpicture}
\end{center}
\end{figure}

The dashed paths have the following structure:

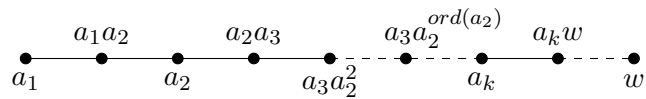
\begin{figure}[h]
\begin{center}
\begin{tikzpicture}

\draw[fill=black]  (0,0) circle (2pt);
\node at (0,-0.3) {$a_1$};
\draw[fill=black]  (1,0) circle (2pt);
\node at (1,0.3) {$a_1a_2$};
\draw[fill=black]  (2,0) circle (2pt);
\node at (2,-0.3) {$a_2$};
\draw[fill=black]  (3,0) circle (2pt);
\node at (3,0.3) {$a_2a_3$};
\draw[fill=black]  (4,0) circle (2pt);
\node at (4,-0.3) {$a_3a_2^2$};
\draw[fill=black]  (5,0) circle (2pt);
\node at (5.5,0.4) {$a_3a_2^{ord(a_2)}$};
\draw[fill=black]  (6,0) circle (2pt);
\node at (6,-0.3) {$a_k$};
\draw[fill=black]  (7,0) circle (2pt);
\node at (7,0.3) {$a_kw$};
\draw[fill=black]  (8,0) circle (2pt);
\node at (8,-0.3) {$w$};

\draw (0,0)--(4,0);
\draw[dashed] (4,0)--(6,0);

\draw (6,0)--(7,0);
\draw[dashed] (7,0)--(8,0);
\end{tikzpicture}
\end{center}
\end{figure}

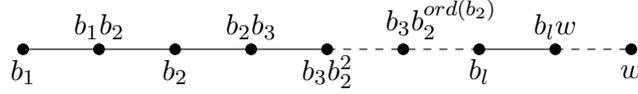
\begin{figure}[h]
\begin{center}
\begin{tikzpicture}

\draw[fill=black]  (0,0) circle (2pt);
\node at (0,-0.3) {$b_1$};
\draw[fill=black]  (1,0) circle (2pt);
\node at (1,0.3) {$b_1b_2$};
\draw[fill=black]  (2,0) circle (2pt);
\node at (2,-0.3) {$b_2$};
\draw[fill=black]  (3,0) circle (2pt);
\node at (3,0.3) {$b_2b_3$};
\draw[fill=black]  (4,0) circle (2pt);
\node at (4,-0.3) {$b_3b_2^2$};
\draw[fill=black]  (5,0) circle (2pt);
\node at (5.5,0.4) {$b_3b_2^{ord(b_2)}$};
\draw[fill=black]  (6,0) circle (2pt);
\node at (6,-0.3) {$b_l$};
\draw[fill=black]  (7,0) circle (2pt);
\node at (7,0.3) {$b_lw$};
\draw[fill=black]  (8,0) circle (2pt);
\node at (8,-0.3) {$w$};

\draw (0,0)--(4,0);
\draw[dashed] (4,0)--(6,0);

\draw (6,0)--(7,0);
\draw[dashed] (7,0)--(8,0);
\end{tikzpicture}
\end{center}
\end{figure}

\newpage
Thus, ${\rm Cay}(G(\Gamma))$ is non-planar.
\end{proof}

\section{Proof of Theorem A}
Now we have all the ingredients to prove Theorem A.
\begin{proof}
Let $\Gamma$ be a graph product graph of finite cyclic groups. Let $\Gamma_{=2}$ be the subgraph generated by its vertices labelled with groups of order $2$ and $\Gamma_{>2}$ the subgraph generated by its vertices labelled with groups of order $>2$.

We first prove that if Theorem A (i), (ii), (iii) or (iv) is not true, then  ${\rm Cay}(G(\Gamma))$ is non-planar.
Assume that $\Gamma_{=2}$ in non-outerplanar, then by \cite[Theorem 1]{Droms} ${\rm Cay}(G(\Gamma_{=2}))$ is non-planar.
By Lemma \ref{subgroup}  ${\rm Cay}(G(\Gamma_{=2}))$ is a subgraph of ${\rm Cay}(G(\Gamma))$ and hence ${\rm Cay}(G(\Gamma))$ is non-planar.
Now we assume that (ii), (iii) or (iv) of Theorem A does not hold. Then it follows from Lemma \ref{subgroup} and Lemma \ref{edge}, Lemma \ref{star} or Lemma \ref{cycle} respectively, that  ${\rm Cay}(G(\Gamma))$ is non-planar.

Assume now, that the graph product graph $\Gamma$ has the properties of Theorem  A (i), (ii), (iii) and (iv). We have to show that ${\rm Cay}(G(\Gamma))$ is planar. The idea of this proof is to build the whole group $G(\Gamma)$ out of subgroups $G(\Gamma')$, whose Cayley graphs are planar, using Lemma \ref{amalgam}. 

For $v\in\Gamma$ we denote by ${\rm lk}(v)$  the subgraph of $\Gamma$ generated by its vertices adjacent to the vertex $v$. We start with the subgroup $G(\Gamma_{=2})$ whose Cayley graph is planar by Theorem A (i). Let $v$ be a vertex in $\Gamma_{>2}$. We denote by $H_1$ the subgraph of $\Gamma$ generated by $\Gamma_{=2}\cup(\left\{v\right\}, \emptyset)$.
Our first goal is to show that ${\rm Cay}(G(H_1))$ is planar.

By the assumptions (ii) and (iii) of Theorem A follows that:
\begin{enumerate}
\item[(a)] ${\rm lk}(v)=(\emptyset, \emptyset)$ or
\item[(b)] ${\rm lk}(v)=(\left\{v_1\right\}, \emptyset)$ with $\varphi(v_1)\cong\Z_2$ or
\item[(c)] ${\rm lk}(v)=(\left\{v_1, v_2\right\}, \emptyset)$ with $\varphi(v_1)\cong\Z_2$ and $\varphi(v_2)\cong\Z_2$
\end{enumerate}

If ${\rm lk}(v)=(\emptyset, \emptyset)$, then we define
$\Gamma_1:=\Gamma_{=2}, \Gamma_2:=(\left\{v\right\}, \emptyset)$. We get $H_1=\Gamma_1\cup\Gamma_2$ and $\Gamma_1\cap\Gamma_2=\emptyset$. Obviously ${\rm Cay}(G(\Gamma_1))$  and ${\rm Cay}(G(\Gamma_2))$ are both planar and the planarity of ${\rm Cay}(G(H_1))$ follows from Lemma \ref{amalgam}.

If ${\rm lk}(v)=(\left\{v_1\right\}, \emptyset)$ with $\varphi(v_1)\cong\Z_2$, then we define
$\Gamma_1:=\Gamma_{=2}, \Gamma_2:=(\left\{v, v_1\right\},\left\{\left\{v, v_1\right\}\right\})$. We get $H_1=\Gamma_1\cup\Gamma_2$ and $\Gamma_1\cap\Gamma_2=(\left\{v_1\right\}, \emptyset)$ with $\varphi(v_1)\cong\Z_2$. It is obvious that ${\rm Cay}(G(\Gamma_1))$ and ${\rm Cay}(G(\Gamma_2))$ are both planar and the planarity of ${\rm Cay}(G(H_1))$ then follows from Lemma \ref{amalgam}.

If ${\rm lk}(v)=(\left\{v_1, v_2\right\}, \emptyset)$ with $\varphi(v_1)\cong\Z_2$ and $\varphi(v_2)\cong\Z_2$, then we define the following subgraphs: 
Let $\widetilde{\Gamma_1}$ be a connected component of $\Gamma_{=2}$ such that $v_1\in\widetilde{\Gamma_1}$ and $\widetilde{\Gamma_2}$ be a connected componen of $\Gamma_{=2}$ such that $v_2\in\widetilde{\Gamma_2}$. Further,  we define $\widetilde{\Gamma_3}:=\Gamma_{=2}-(\widetilde{\Gamma_1}\cup\widetilde{\Gamma_2})$. The connected components $\widetilde{\Gamma_1}$ and $\widetilde{\Gamma_2}$ are disjoint, since $\Gamma$ has by assumption property (iv) of Theorem A.  
Now we define the following graphs: $\Gamma_1:=\widetilde{\Gamma_1}\cup\widetilde{\Gamma_3}$ and $\Gamma_2$ be the subgraph of $\Gamma$ generated by $\widetilde{\Gamma_2}$ and $\left\{v, v_1\right\}$. Applying Lemma \ref{amalgam} and Lemma \ref{path} it follows that ${\rm Cay}(G(\Gamma_2))$ is planar. Furthermore we get $H_1=\Gamma_1\cup\Gamma_2$ and  $\Gamma_1\cap\Gamma_2=(\left\{v_1\right\}, \emptyset)$ with $\varphi(v_1)\cong\Z_2$. The graph ${\rm Cay}(G(\Gamma_1))$ is as a subgraph of ${\rm Cay}(G(\Gamma_{=2}))$ planar and  ${\rm Cay}(G(\Gamma_2))$ is also planar. The planarity of ${\rm Cay}(G(H_1))$ follows from Lemma \ref{amalgam}.

Since we can repeat this for each remaining vertex $v\in\Gamma_{>2}$, the graph ${\rm Cay}(G(\Gamma))$ is planar. 
\end{proof}

\section{Proof of Proposition B}
We turn now to the proof of Proposition B.
\begin{proof}
Let $H\subseteq G$ be a finitely generated subgroup of $G$. Since $G$ has a planar Cayley graph it follows from \cite{Babai}, \cite[Theorem 2.2]{DromsServatius} that $H$ has also a planar Cayley graph. Further, it was proven in \cite[Theorem 5.1]{Droms2} that if a group has a planar Cayley graph then this group is finitely presented. Thus, $H$ is finitely presented.
\end{proof}

We want to remark, that coherence of graph products was also studied in \cite{Varghese}.

\subsection*{Acknowledgements} The author would like to thank the referee for many helpful comments.

\end{document}